\documentclass[11pt,reqno]{amsart}
\usepackage{amsmath}
\usepackage{amssymb}
\usepackage{amsfonts}
\usepackage{graphicx}
\usepackage{amsthm}
\usepackage{enumerate}
\usepackage{lscape}
\usepackage{dsfont}
\usepackage{color}

\newcommand{\CP}{\mathds{C}\mathrm{P}}

\newcommand{\CH}{\mathds{C}\mathrm{H}}

\newcommand{\ol}{\mathrm{Hol}}
                       
\newcommand{\f}{\rightarrow}                  
\newcommand{\C}{\mathds{C}}            
\newcommand{\de}{\partial}

\newtheorem{theor}{Theorem}

\newtheorem{lem}[theor]{Lemma}

\begin{document}

\title{Berezin--Engli\v{s}' quantization of Cartan--Hartogs domains}
\author[M. Zedda]{Michela Zedda}
\address{Dipartimento di Matematica ``Giuseppe Peano'' Universit\`{a} di Torino,
Via Carlo Alberto 1, 09124 Torino, Italy}
\email{michela.zedda@gmail.com}
\thanks{
The author was supported by the project FIRB ``Geometria Differenziale e teoria geometrica delle funzioni'' and by G.N.S.A.G.A. of I.N.d.A.M.}
\date{}

\begin{abstract}
We prove the existence of a  Berezin-Engli\v{s} quantization for Cartan--Hartogs domains.  
\end{abstract}

\maketitle
\section{Introduction and statement of the main result}

Let $(M, \omega)$ be a symplectic manifold and let $\{\cdot, \cdot \}$ be the associated Poisson bracket.
A {\em Berezin quantization} (we refer to \cite{Ber1} for details) on $M$ is given by a family of associative algebras $\mathcal A_\hbar$,
where the parameter $\hbar$ (which plays the role of Planck constant) ranges over a set $E$ of positive reals with limit point $0$, such that in the direct sum $\oplus_{h \in E} \mathcal{A}_h$ with component-wise product $*$, there exists a subalgebra ${\mathcal A}$ satisfying the following properties:
\begin{itemize}
\item[$(i)$] for any element 
$f=f(\hbar)\in {\mathcal A}$, where $f(\hbar)\in {\mathcal A}_\hbar$, there exists a limit $\lim_{\hbar\rightarrow 0} f(\hbar)=\varphi (f)\in C^{\infty}(\Omega)$,
\item[$(ii)$]
for $f, g \in {\mathcal A}$
$$\varphi (f*g)=\varphi (f)\varphi (g),\    \    \  \varphi\left(\hbar^{-1}(f*g-g*f)\right)=i\{\varphi(f), \varphi (g)\},$$
\item[$(iii)$]
for any pair of points $x_1,x_2 \in \Omega$ there exists $f \in \mathcal{A} $ such that $\varphi (f) (x_1) \neq \varphi(f)(x_2)$.
\end{itemize}

Consider now a real analytic noncompact K\"ahler manifold $M$ endowed with a K\"ahler metric $g$ and let $\Phi$ be a (real analytic) K\"ahler potential for $g$, i.e. in a neighborhood of a point $p\in M$ the K\"ahler form $\omega$ associated to $g$ can be written $\omega=\frac{i}{2}\partial\bar \partial \Phi$.
One can extend $\Phi$ to a sesquianalytic function $\Phi(x,\bar  y)$ on a neighborhood $U$ of the diagonal of $M \times M$, in such a way that $\Phi(x,\bar  x)=\Phi(x)$, and define the {\em Calabi's diastasis function} $D_g$ on $U$ by:
\begin{equation}\label{diastasis}
 D_g(x,y)=\Phi(x,\bar  x) + \Phi(y,\bar  y) - \Phi(x,\bar  y) - \Phi(y,\bar  x).
\end{equation}
Observe that $D_g$ is independent from the potential chosen, which is defined up to the addition with the real part of a holomorphic function.
Moreover,  it is easily seen that $D_g$ is real-valued,  symmetric in $x$ and $y$ and $D_g(x, x)=0$ (see \cite{Cal} for details and further results).

For $\alpha>0$ consider the weighted Bergman space $\mathcal{H}_\alpha$ of square integrable holomorphic functions on $M$ with respect to the measure $e^{-\alpha\Phi}\frac{\omega^n}{n!}$, i.e. $f$ belongs to $\mathcal{H}_\alpha$ iff $\int_Me^{-\alpha\Phi}|f|^2\frac{\omega^n}{n!}<\infty$. Define the $\epsilon$-function associated to $g$ to be the function:
\begin{equation}\label{epsilon}
\epsilon_{\alpha g}(x)=e^{-\alpha\Phi(x)}K_{\alpha}(x,x), \quad x\in M,
\end{equation}
where $K_\alpha(x, x)$ is the reproducing kernel of $\mathcal{H}_\alpha$, i.e. $K_\alpha(x, x)=\sum_j f^\alpha_j(x)\overline{ f^\alpha_j(x)}$, for an orthonormal basis $\{f^\alpha_j\}$ of $\mathcal{H}_\alpha$. As suggested by the notation it is not difficult to verify that $\epsilon_{\alpha g}$ depends only on the metric $g$ and not on the choice of the K\"ahler potential $\Phi$ or on the orthonormal basis. In the literature the function $\epsilon_{\alpha g}$ was first introduced under the name of $\eta$-{\em function} by J. Rawnsley in \cite{rawnsley}, later renamed as $\theta$-{\em function} in \cite{cgr1}.

In  \cite{Ber1} F. A. Berezin was able to establish a quantization procedure on $(M, \omega)$ under the following conditions:
\begin{itemize}
\item[(A)] the function $e^{-D_g(x, y)}$  is globally defined on $M\times M$,  $e^{-D_g(x, y)}\leq1$ and $e^{-D_g(x, y)}=1$ if and only if $x=y$;
\item[(B)] for large enough $\alpha$, the function $\epsilon_{\alpha g}$ is a positive constant (depending only on $\alpha$).
\end{itemize}
These conditions are satisfied for example by bounded symmetric domains \cite{Ber1} and by all homogeneous bounded domains (see the recent paper \cite{loiberezin}).

\vspace{0.5cm}
\noindent{\bf Remark.}
Notice that Condition (B) can be expressed by saying that the K\"ahler metric $g$ is {\em balanced} for large enough $\alpha$. The definition of balanced metrics has been introduced by Donaldson \cite{donaldson} for algebraic manifolds and by C. Arezzo and A. Loi \cite{arezzoloi} in the noncompact case. Observe also that balanced metrics are strictly related to projectively induced metrics, i.e. those K\"ahler metrics $g$ on a complex manifold $M$, such that there exists a holomorphic and isometric immersion $F\!:M\rightarrow \mathds{C}{\rm P}^N$, $N\leq \infty$, $F^*(g_{FS})=g$, where  $g_{FS}$ is the Fubini-Study metric on $\mathds{C}{\rm P}^N$, i.e. the metric whose K\"ahler form $\omega_{FS}$ in homogeneous
coordinates $[Z_0,\dots, Z_{N}]$ reads as
$\omega_{FS}=\frac{i}{2}\partial\bar\partial\log \sum_{j=0}^{N}
|Z_j|^2$. In fact, if $\epsilon_{\alpha g}$ is constant then the map $F_\alpha\!: M\rightarrow \mathds{C}{\rm  P}^N$, $N\leq \infty$, $F_\alpha=[f^\alpha_0,\dots,f_N^\alpha]$, by:
\begin{equation}\label{balprojind}
\begin{split}
F_\alpha^*\omega_{FS}=&\frac{i}{2}\de\bar\de\log\sum_{j=0}^N|f^\alpha_j(z)|^2
=\frac{i}{2}\de\bar\de\log K_{\alpha} (z, \bar z)=\\
=&\frac{i}{2}\de\bar\de\log \epsilon_{\alpha g}+\frac{i}{2}\de\bar\de\log e^{\alpha\Phi}
=\frac{i}{2}\de\bar\de\log \epsilon_{\alpha g}+   \,\alpha\omega,
\end{split}\nonumber
\end{equation}
is an holomorphic and isometric immersion.
Observe finally that in the joint work with A. Loi \cite{balancedCH}, the author of the present paper proved that a Cartan--Hartogs domain is not balanced unless it is the complex hyperbolic space.\\

Berezin's seminal paper has inspired several interesting papers both from the mathematical and physical point of view (see \cite{cgr1},  \cite{cgr2},  
\cite{cgr3}, \cite{cgr4} for a quantum geometric interpretation of Berezin quantization  and its extension to the compact case). In \cite{englis} M. Engli\v{s} extended Berezin's method to complex domains satisfying condition (A) and such that their $\epsilon$-function is not necessarily  constant, but only satisfies the following weaker  asymptotic condition:

\begin{itemize}
\item[(${\rm B}'$)]  the $\epsilon$-function (\ref{epsilon}) admits a sesquianalytic extension on $M\times M$
\begin{equation}
\epsilon_{\alpha g}(x,\bar y):=e^{-\alpha\Phi(x,\bar y)}K_{\alpha}(x,\bar y)\nonumber
\end{equation}
and  there exists a infinite set $E$ of integers such that for all $\alpha\in E$, $x$, $y\in M$,
$$\epsilon_{\alpha g}(x,\bar y)=e^{-\alpha\Phi(x,\bar y)}K_{\alpha}(x,\bar y)=\alpha^n+B(x,\bar y)\alpha^{n-1}+C(\alpha,x,\bar y)\alpha^{n-2},$$
where  $B(x,\bar y)$ and $C(\alpha,x,\bar y)$ are sesquianalytic  functions in $x$ and $y$ which satisfy:
$$\mathrm{sup}_{x,y\in M} |B(x,\bar y)|<+\infty,\quad \mathrm{sup}_{x,y\in M, \alpha\in E}|C(\alpha,x,\bar y)|<+\infty.$$
\end{itemize}

We refer the reader to  \cite{englis} for various examples   of complex  domains in $\C^n$ satisfying  conditions (A) and (${\rm B}'$)  and so admitting  a Berezin quantization.\\

This paper deals with a $1$-parameter family of domains, called {\em Cartan--Hartogs domains}, defined as follows.
Let $\Omega\subset \C^d$ be a Cartan domain, i.e. an irreducible bounded symmetric domain, of complex dimension $d$ and genus $\gamma$. For all positive real numbers $\mu$ define a Cartan-Hartogs domain by:
\begin{equation}\label{defm}
M_{\Omega}(\mu)=\left\{(z,w)\in \Omega\times\C,\ |w|^2<N_\Omega(z,\bar z)^\mu\right\},
\end{equation}
where $N_\Omega(z,\bar z)$ is the  {\em generic norm} of $\Omega$, i.e.
\begin{equation}\label{genericnorm}
N_\Omega(z, \bar z)=(V(\Omega)K(z, \bar z))^{-\frac{1}{\gamma}},
\end{equation}
where $V(\Omega)$ is the total volume of $\Omega$ with respect to the Euclidean measure of the ambient complex Euclidean space and $K(z, z)$ is its Bergman kernel.
Consider on $M_{\Omega}(\mu)$ the metric $g(\mu)$  whose associated K\"ahler form $\omega(\mu)$ can be described by the (globally defined)
K\"ahler potential centered at the origin
\begin{equation}\label{diastM}
\Phi(z,w)=-\log(N_{\Omega}(z,\bar z)^\mu-|w|^2).
\end{equation}
The domain $\Omega$ is called  the {\em  base} of the Cartan--Hartogs domain 
$M_{\Omega}(\mu)$ (one also  says that 
$M_{\Omega}(\mu)$  is based on $\Omega$).

These domains have been considered by several authors (see e.g. \cite{roos} and references therein). In \cite{roos} the authors show that 
for $\mu_0=\gamma/(d+1)$, $(M_{\Omega}(\mu_0),g(\mu_0))$ is a complete K\"ahler-Einstein manifold which is homogeneous if and only if $\Omega$ is the complex hyperbolic space.
In \cite{articwall} the author of the present paper jointly with A. Loi proved that for $\Omega\neq\CH^d$, the metric $\alpha g(\mu)$ on $M_\Omega(\mu)$ is projectively induced for all positive real number $\alpha\geq \frac{(r-1)a}{2\mu}$, where $r$ is the rank of $\Omega$ and $a$ is one of its invariants, exhibing the first example complete, noncompact, nonhomogeneous and projectively induced K\"ahler-Einstein metric. 
 In \cite{zedda} the author of the present paper proved that $g(\mu)$ is extremal (in the sense of Calabi \cite{Calabi82}) if and only if it is K\"ahler--Einstein, and that if the coefficient $a_2$ of Engli\v{s} expansion (cfr. \cite{englisasymp}) of the $\epsilon$-function associated to $g(\mu)$ is constant, then it is K\"ahler--Einstein, conjecturing also that: {\em the coefficient $a_2$ of Engli\v{s} expansion of the $\epsilon$-function associated to  $g(\mu)$ is constant iff $(M_\Omega(\mu), g(\mu))$ is biholomorphically isometric to the complex hyperbolic space}. This conjecture has been recently proved by Z. Feng and Z. Tu in \cite{fengtu}, where they also obtain an explicit formula for the Bergman kernel of the weighted Hilbert space $\mathcal{H_\alpha}$ and for the $\epsilon$-function associated to $(M_\Omega(\mu), g(\mu))$. \\

The aim of this paper is to prove the following result:
\begin{theor}\label{main}
Let $\Omega$ be a Cartan domain of (complex) dimension $d$ and let $\mu\in W(\Omega)$ and $\alpha>d+1$. Then the Cartan-Hartogs domain $(M_\Omega(\mu),\alpha g(\mu))$ admits a Berezin quantization.
\end{theor}
Here $W(\Omega)$ is the {\em Wallach set} associated to $\Omega$, which consists of all $\eta\in\C$ such that there exists a Hilbert space ${\mathcal{H}}_\frac\eta\gamma$ whose reproducing kernel is   $K_{\frac{\eta}{\gamma}}$ (we refer the reader to \cite{arazy}, \cite{faraut} and \cite{upmeier} for more details and results). 
It turns out (see Corollary $4.4$ p. $ 27$ in  \cite{arazy} and references therein)
that $W(\Omega)$ consists only of real numbers and depends   on two of the domain's invariants, denoted  by  $a$ (strictly positive real number) and  $r$
(the rank of $\Omega$).
More precisely we have
\begin{equation}\label{wallachset}
W(\Omega)=\left\{0,\,\frac{a}{2},\,2\frac{a}{2},\,\dots,\,(r-1)\frac{a}{2}\right\}\cup \left((r-1)\frac{a}{2},\,\infty\right).
\end{equation}

\vspace{1cm}

The next section is dedicated to the proof of Theorem \ref{main}. The proof is based on the result in  \cite{articwall} mentioned above and on the explicit expression of the $\epsilon$-function associated to $(M_{\Omega}(\mu),\alpha g(\mu))$ given by Z. Feng and Z. Tu in \cite{fengtu}. \\

The author would like to thank Andrea Loi for the useful comments and discussions.

\section{Proof of the main results}

In his seminal paper \cite{Cal} Calabi gives necessary and sufficient conditions for a $n$-dimensional K\"ahler manifold $(M,g)$ to admit a holomorphic and isometric immersion into a complex space form, in terms of the diastasis function \eqref{diastasis}. In particular, we recall here the following result, needed in the proof of Lemma \ref{condA} below.  
\begin{theor}[E. Calabi]\label{Calabi}
Set homogeneous coordinates $[Z_0:\cdots:Z_j:\cdots]$ in $\CP^{\infty}$, let $U_0= \{Z_0\neq 0\}$ and let  $f:(M,g)\rightarrow \mathds{C}{\rm P}^{\infty}$ be an holomorphic and isometric immersion, i.e. $f^*g_{FS}=g$. Then the metric $g$ is real analytic and we have:
$$D_g=D_{FS}\circ f :M\setminus f^{-1}(H_0)\times M\setminus f^{-1}(H_0)\rightarrow \mathds{R},$$
where $H_0=\mathds{C}{\rm P}^{\infty}\setminus U_0$.
\end{theor}
Observe that if $p$, $p'\in\CP^\infty$ are points with homogeneous coordinates $[Z_0:\cdots :Z_j:\cdots]$ and $[Z'_0:\cdots :Z'_j:\cdots]$ respectively, then we have (cfr. \cite[Eq. (29)]{Cal}):
\begin{equation}\label{diastfs}
D_{FS}(p,p')=\log\frac{\sum_{j=0}^\infty |Z_j|^2\sum_{j=0}^\infty |Z'_j|^2}{\sum_{j=0}^\infty Z_j\bar Z'_j}.
\end{equation}
In order to prove Theorem \ref{main} we need the following four lemma:
\begin{lem}\label{condA}
Let $(M, g, \omega)$ be a noncompact complete K\"ahler manifold. Assume the metric $g$ is projectively induced through an injective map $f$ such that $f(M)\subset\ell^2(\mathds{C})$.
Then $(M, \omega)$ satisfies Condition ${\rm( A)}$ given above. 
\end{lem}
\begin{proof}
Let $f\!\colon M\to\CP^\infty$ be a holomorphic immersion such that $f^*\omega_{FS}=\omega$ and $f(M)\subset\ell^2(\mathds{C})\subset\CP^\infty$.
Then by Theorem \ref{Calabi} above, if $D_g$ is the diastasis function of $(M,g)$ and $D_{FS}$ is the one associated to the Fubini--Study metric on $\CP^\infty$ we have
$$D_g(x,y)=D_{FS}(f(x),f(y)), \quad \forall\, x,y\in M.$$
Further, since $f(M)\subset \ell^2(\mathds{C})$, we can assume $f^{-1}(H_0)=\emptyset$, i.e. $D_g(x,y)$ is defined on the whole $M\times M$. Further, by the expression of $D_{FS}$ \eqref{diastfs},
it follows by Cauchy-Schwartz's inequality that $e^{-D_{FS}(f(x),f(y))}\leq 1$. Finally, from $D_{FS}(p,p')=0$ iff $p=p'$ and the injectivity of $f$, follows $e^{-D_g(x,y)}=1$ iff $x=y$. Thus condition (A) above is fulfilled.
\end{proof}
\begin{lem}\label{injective}
The holomorphic and isometric immersion $f\!:M_\Omega(\mu)\f \CP^\infty$, $f^*g_{FS}=g(\mu)$, when exists, is injective and such that $f(M_\Omega(\mu))\subset\ell^2(\mathds{C})$.
\end{lem}
\begin{proof}
Let $f\!:M_\Omega(\mu)\f \CP^\infty$ be a holomorphic map such that $f^*\omega_{FS}=\omega(\mu)$. According to \cite[Lemma 8]{balancedCH} up to unitary transformation of $\CP^\infty$ we have:
$$
f(z,w)=\left[ 1, s, h_{\frac{\mu\, \alpha}{\gamma}},\dots,\sqrt{\frac{(m+ \alpha-1)!}{(\alpha-1)!m!}}h_{\frac{\mu(\alpha +m)}{\gamma}}w^m,\dots\right],
$$
where $s=(s_1,\dots, s_m,\dots)$ with:
$$s_m=\sqrt{\frac{(m+ \alpha-1)!}{(\alpha-1)!m!}}w^m,$$
and $h_k=(h_k^1,\dots,h_k^j,\dots)$ denotes the sequence of holomorphic maps on $\Omega$ such that the immersion $\tilde h_k=(1,h_k^1,\dots, h_k^j,\dots)$, $\tilde h_k\!:\Omega\f\CP^\infty$, satisfies $\tilde h_k^*\omega_{FS}=k\, \omega_B$ (where $\omega_B$ is the Bergman metric on $\Omega$), i.e. 
$$
1+\sum_{j=1}^{\infty}|h_k^j|^2=\frac{1}{N^{\gamma\, k}}.
$$
The injectivity of $f$ follows from that of $ h_{\frac{\mu\, \alpha}{\gamma}}$ (see \cite[Lemma 2.1]{loi06}) and noticing that $s_1=\sqrt\alpha \,w$. Further, $f_0=1$ implies $f^{-1}(H_0)=\emptyset$ in the notation of Theorem \ref{Calabi} above, and thus $f(M_\Omega(\mu))\subset\ell^2(\mathds{C})$, as wished.
\end{proof}
\begin{lem}\label{limite}
Let $(M_\Omega(\mu),g(\mu))$ be a Cartan--Hartogs domain. Then the following holds true:
$$  \mathrm{sup}_{z,z',w,w'\in M_\Omega(\mu)} \left|1-w\bar w'N_\Omega(z, \bar z')^{-\mu}\right|<+\infty.$$ 
\end{lem}
\begin{proof}
Observe first that we need only to check the case when $z$ or $z'\to \partial \Omega$, as it follows by the definition of generic norm and considering that if we fix $z$, $z'\in \Omega$, then the following:
$$\mathrm{sup}_{w,w'\in M_\Omega(\mu)} \left|N_\Omega(z,\bar  z')^\mu-w\bar w'\right|<+\infty,$$
is always satisfied, since by Condition (A):
$$N_\Omega(z, \bar z')^\mu-w\bar w'=e^{-\Phi(z,w,\bar z',w')}\leq 1.$$
Thus, assume that $z\to \partial \Omega$. Then by definition of $M_\Omega(\mu)$ \eqref{defm}  we have $|w|^2\to N_\Omega(z,\bar  z)^\mu$.
Further, since $\mu\in W(\Omega)$ then $K_{\frac\mu\gamma}(z,\bar  z')$ is the reproducing kernel of the Hilbert space $\mathcal{H}_{\frac\mu\gamma}$. Thus, if $\{f_j\}_{j=0,1,\dots}$ is an orthonormal basis of $\mathcal{H}_{\frac\mu\gamma}$, we have:
$$K_{\frac\mu\gamma}(z, \bar z')=\sum_{j=0}^\infty f_j(z)\overline{f_j(z')},$$
and by Cauchy-Schwarz's inequality it follows:
$$
ww'K_{\frac\mu\gamma}(z,\bar  z')\leq |w|^2|w'|^2K_{\frac\mu\gamma}(z, \bar z)K_{\frac\mu\gamma}(z', \bar z').
$$
i.e., by the definition of generic norm \eqref{genericnorm}:
$$
ww'N_\Omega(z,\bar  z')^{-\mu}\leq c |w|^2|w'|^2N_\Omega(z, \bar z)^{-\mu}N_\Omega(z', \bar z')^{-\mu},
$$
where $c=V(\Omega)^{\frac1\gamma}$. Thus
$$\left|w\bar w'N_\Omega(z,\bar  z')^{-\mu}\right|\leq|w|^2 |w'|^2\left|cN_\Omega(z', \bar z')^{-\mu}N_\Omega(z, \bar z)^{-\mu}\right|\to|w'|^2\left|c N_\Omega(z', \bar z')^{-\mu}\right|,$$
and we are done.
\end{proof}

Observe that it follows from the computation in \cite{balancedCH} that the weighted Hilbert space:
\begin{equation}\label{hilbCH}
\mathcal{H}_{\alpha}=\left\{ \varphi\in\ol(M_{\Omega}(\mu)) \ | \ \, \int_{M(\mu)}\left(N^\mu_\Omega-|w|^2\right)^\alpha|\varphi|^2\frac{\omega( \mu )^{d+1}}{(d+1)!}<+\infty\right\},
\end{equation}
is not trivial for all $\alpha>d+1$. For completeness, we give here a semplified proof  of this fact in the following lemma:
\begin{lem}\label{existH}
The weighted Hilbert space $\mathcal{H}_\alpha$ given in \eqref{hilbCH} is not trivial for all $\alpha>d+1$.
\end{lem}
\begin{proof}
It is enough to prove that $1\in \mathcal{H}_\alpha$ for all $\alpha>d+1$. i.e. that:
$$
\int_{M(\mu)}\left(N^\mu_\Omega-|w|^2\right)^\alpha\frac{\omega( \mu )^{d+1}}{(d+1)!}<+\infty
$$
Observe first that up to the multiplication with a positive constant:
 $$\frac{\omega(\mu)^{d+1}}{(d+1)!}=\frac{1}{(N^{\mu}_\Omega-|w|^2)^{d+2}} \frac{\omega_0^{d+1}}{(d+1)!},$$
as it follows by a long but straightforward computation of the determinant of the metric $g(\mu)$.
Thus, we need to prove that:
$$
\int_{M_{\Omega}(\mu)}(N^{\mu}_\Omega-|w|^2)^{\alpha-(d+2)}\frac{\omega_0^{d+1}}{(d+1)!}<+\infty.
$$
Setting polar coordinates we get:
$$
\int_{M_{\Omega}(\mu)}(N^{\mu}_\Omega-|w|^2)^{\alpha-(d+2)}\frac{\omega_0^{d+1}}{(d+1)!}=\frac{\pi}{(d+1)!}\int_\Omega\int_0^{N^{\mu}_\Omega}(N^{\mu}_\Omega-\rho)^{\alpha-(d+2)}d\rho\,\omega_0^{d}.
$$
The integral:
$$
\int_0^{N^{\mu}_\Omega}(N^{\mu}_\Omega-\rho)^{\alpha-(d+2)}d\rho,
$$
is convergent iff $\alpha-(d+2)>-1$, i.e. iff $\alpha>d+1$. Setting $\alpha>d+1$ we get:
$$
\int_{M_{\Omega}(\mu)}(N^{\mu}_\Omega-|w|^2)^{\alpha-(d+2)}\frac{\omega_0^{d+1}}{(d+1)!}=\frac{\pi}{(d+1)!}\frac{1}{\alpha-d-1}
\int_\Omega\,N^{\mu_0(\alpha-d-1)}_\Omega \omega_0^d
$$
and by \cite[Prop. 2.1, p. 358]{roosformula} the integral on the right hand side is convergent whenever  $\mu(\alpha-d-1)>-1$, i.e. for $\alpha>1+d-\frac1\mu$.
\end{proof}

We are now in the position of proving our main theorem.

\begin{proof}[Proof of Theorem \ref{main}]
By the discussion at the begin of this paper we need to prove that Condition ${\rm (A)}$ and ${\rm (B')}$ are fulfilled. Set $\alpha>d+1$
 and let $a$, $b$ be the two geometrical invariants of $\Omega$, $r$ and $\gamma$ respectively its rank and its genus. In order to apply Lemma \ref{condA} and prove that Condition (A) holds true, observe that by \cite[Th. 2]{articwall}, $(M_\Omega(\mu),\alpha g(\mu))$ is projectively induced for all $\alpha\geq \frac{(r-1)a}{2\mu}$, which is always satisfied for $\alpha>d+1$ and $\mu\in W(\Omega)$. In fact by \eqref{wallachset} $\mu\geq \frac a2$, i.e. $\frac{(r-1)a}{2\mu}\leq r-1$, and since the dimension $d$ is related to $a$, $b$ and $r$ by the formula $d=\frac{r(r-1)}2a+r\,b+r,$
 we also have $r-1<d$. Further, the injectivity of the map $f\!:M_\Omega(\mu)\f \CP^\infty$ and the condition $f(M_\Omega(\mu))\subset\ell^2(\mathds{C})$ are guaranteed by Lemma \ref{injective}.

In order to prove that Condition ${\rm (B')}$ holds true, let $E$ be the set of all integers greater than $d+1$. For $\alpha\in E$ by Lemma \ref{existH} the Hilbert space $\mathcal{H}_\alpha$ defined in \eqref{hilbCH} is not trivial and, as proven in \cite[Th. 3.1]{fengtu}, $\epsilon_{\alpha g(\mu)}$ reads:
$$
\epsilon_{\alpha g(\mu)}(z,w)=\frac 1{\mu^d}\sum_{k=0}^d\frac{D^k\tilde \chi(d)}{k!}\left(1-\frac{||w||^2}{N_\Omega(z,\bar z)^\mu}\right)^{d-k}\frac{(\alpha-d+k-1)!}{(\alpha-d-2)!},
$$
for
$$
D^k\tilde \chi(d)=\sum_{j=0}^k{ k\choose j}(-1)^j\tilde \chi(d-j)
$$
and
$$
\tilde \chi(d-j)= \prod_{j=1}^r\frac{\Gamma(\mu (d-j)-\gamma+1+(j-1)\frac a2+1+b+(r-j)a)}{\Gamma(\mu (d-j)-\gamma+1+(j-1)\frac a2)},
$$
where $\Gamma$ is the usual $\Gamma$-function.
Observe that both the potential $\Phi(z,w)$ given in \eqref{diastM} and the reproducing kernel of $\mathcal{H}_\alpha$, admit a sesquianalytic extension on $M_\Omega(\mu)\times M_\Omega(\mu)$. Thus, it follows from \eqref{epsilon} that also $\epsilon_{\alpha g(\mu)}$ does and in particular it reads:
\begin{equation}\label{epsilonCH}
\epsilon_{\alpha g(\mu)}(z,w,z',w')=\frac 1{\mu^d}\sum_{k=0}^d\frac{D^k\tilde \chi(d)}{k!}\left(1-\frac{w\bar w'}{N_\Omega(z, \bar z')^\mu}\right)^{d-k}\frac{(\alpha-d+k-1)!}{(\alpha-d-2)!}.
\end{equation}
Since by \cite[Lemma 3.3, 3.4, 3.5]{fengtu} we have:
$$\frac{D^d\tilde \chi(d)}{d!}=\mu^d,\qquad
\frac{D^{d-1}\tilde \chi(d)}{(d-1)!}=\mu^{d-1}\frac{d\,(\mu(d+1)-\gamma)}2,
$$
it follows that we can write:
\begin{equation}\label{epsilonexp}
\epsilon_{\alpha g(\mu)}(z,w,\bar z',\bar w')=\alpha^{d+1}+B(z,w,\bar z',\bar w')\alpha^d+C(\alpha,z,w,\bar z',\bar w')\alpha^{d-1},
\end{equation}
with: $$
B(z,w,\bar z',\bar w')=-\frac{(d+1)(d+2)}2+\frac{d\,(\mu(d+1)-\gamma)}{2\mu}\left(1-\frac{w\bar w'}{N_\Omega(z, z')^\mu}\right).
$$
By Lemma \ref{limite}, since $B(z,w,\bar z',\bar w')$ depends only on $z,w,z',w'$ through $1-w\bar w'N_\Omega(z, z')^{-\mu}$, we have:
$$\mathrm{sup}_{z,z',w,w'\in M_\Omega(\mu)} |B(z,w,\bar z',\bar w')|<+\infty.$$
Thus, it remains to show that:
$$
\mathrm{sup}_{z,z',w,w'\in M_\Omega(\mu), \alpha\in E}|C(\alpha,z,w,\bar z',\bar w')|<+\infty.
$$
By \eqref{epsilonexp} we have:
$$
C(\alpha,z,w,\bar z',\bar w')=\left(\epsilon_{\alpha g}(z,w,\bar z',\bar w')-\alpha^{d+1}-B(z,w,\bar z',\bar w')\alpha^d\right)\alpha^{-(d-1)},
$$
where by \eqref{epsilon}, $\epsilon_{\alpha g}(z,w,\bar z',\bar w')-\alpha^{d+1}-B(z,w,\bar z',\bar w')\alpha^d$ is a polynomial of degree $\alpha^{(d-1)}$. Thus:
$$
\mathrm{sup}_{\alpha\in E}|C(\alpha,z,w,\bar z',\bar w')|<+\infty.
$$
The convergence for $z,z',w,w'\in M_\Omega(\mu)$ follows by noticing that the expression \eqref{epsilonCH} presents a finite sum of  factor depending by $z,w, z', w'$ only through $1-w\bar w'N_\Omega(z,\bar z')^{-\mu}$, which is bounded by Lemma \ref{limite}.
\end{proof}

\noindent{\bf Remark.}
Observe that according to \cite{englis}, the expression of $B(z,w,\bar z',\bar w')$ in the proof of Theorem \ref{main} is actually one over half the scalar curvature of $(M_{\Omega}(\mu),\alpha g(\mu))$ (see \cite{zedda} for a proof).


\begin{thebibliography}{99}

\bibitem{arazy} J. Arazy, \emph{A Survey of Invariant Hilbert Spaces of Analytic Functions on Bounded Symmetric Domains}, Contemporary Mathematics 185 (1995).

\bibitem{arezzoloi} C. Arezzo, A. Loi, {\em Moment maps, scalar curvature and quantization of K\"ahler manifolds}, Comm. Math. Phys. 246 (2004), no. 3, 543--559. 



\bibitem{Ber1} F. A. Berezin, {\em
Quantization},
Izv. Akad. Nauk SSSR Ser. Mat. 38 (1974), 1116--1175 (Russian).





\bibitem{cgr1} M. Cahen, S. Gutt, J. H. Rawnsley,
{\em Quantization of K\"{a}hler manifolds I: Geometric
interpretation of Berezin's quantization}, JGP. 7 (1990), 45-62.


\bibitem{cgr2} M. Cahen, S. Gutt, J. H. Rawnsley,
{\em Quantization of K\"{a}hler manifolds II}, Trans. Amer. Math.
Soc. 337 (1993), 73-98.

\bibitem{cgr3} M. Cahen, S. Gutt, J. H. Rawnsley,
{\em Quantization of K\"{a}hler manifolds  III}, Lett. Math. Phys.
30 (1994), 291-305.

\bibitem{cgr4} M. Cahen, S. Gutt, J. H. Rawnsley,
{\em Quantization of K\"{a}hler manifolds IV}, Lett. Math. Phys. 34
(1995), 159-168.


\bibitem{Cal} E. Calabi,  \emph{Isometric Imbedding of Complex Manifolds}, Ann. of Math. 58 (1953), 1--23.

\bibitem{Calabi82} E. Calabi, \emph{Extremal K\"ahler metrics}, In Seminar on Differential Geometry vol. 16 of 102 (1982), Ann. of Math. Stud., Princeton University Press, 259-290.








\bibitem{donaldson} S. Donaldson, {\em Scalar curvature and projective embeddings}, I. J. Diff. Geom. 59 (2001), 479--522.

\bibitem{englis} M. Engli\v{s} 
\emph{Berezin Quantization and Reproducing Kernels on Complex Domains},
Trans. Amer. Math. Soc. vol. 348 (1996), 411-479.



\bibitem{englisasymp} M. Engli\v{s}, {\em The asymptotics of a Laplace integral on a K\"ahler manifold}, J. Reine Angew. Math. 528 (2000) 1--39. 

\bibitem{faraut} J. Faraut, A. Koranyi, \emph{Function Spaces and Reproducing Kernels on Bounded Symmetric Domains}, Journal of Functional Analysis 88 (1990), 64--89.

\bibitem{fengtu} Z. Feng, Z. Tu, {\em On canonical metrics on Cartan-Hartogs domains}, 	arXiv:1403.7975 [math.CV] (to appear in Math. Z.).




%
%




\bibitem{loi06} A. Loi, {\em Calabi's diastasis function for Hermitian symmetric spaces}, Diff. Geom. Appl. 24 (2006), 311-319.

\bibitem{loiberezin} A. Loi, R. Mossa, {Berezin quantization of homogeneous bounded domains}, Geom. Dedicata 161 (2012) 119-128.


\bibitem{balancedCH} A. Loi, M. Zedda, {\em Balanced metrics on Cartan and Cartan--Hartogs domains}, Math. Z. 270 (2012), no. 3-4, 1077--1087.

\bibitem{articwall} A. Loi, M. Zedda, \emph{K\"ahler--Einstein  submanifolds of  the infinite dimensional  projective space}, Math. Ann. 350 (2011), 145--154.



%







\bibitem{rawnsley} J. Rawnsley, \emph{Coherent states and K\"ahler manifolds}, Quart. J. Math. Oxford (2), n. 28 (1977), 403--415.

\bibitem{roosformula} G. Roos, Keiping Lu, Weiping Yin,
\emph{New classes of domains with explicit Bergman kernel}, Science in China 47, no. 3 (2004), 352--371.

\bibitem{roos} G. Roos, A. Wang, W. Yin, L. Zhang \emph{The K\"ahler-Einstein metric for some Hartogs
domains over bounded symmetric domains}, Science in China, vol $49$, September
2006, pp. 1175-1210 .

%




\bibitem{upmeier}
H. Upmeier,
{\em Index theory for Toeplitz operators on bounded symmetric domains},
 Repr\'esentations des groupes et analyse complexe (Luminy, 1986), 89--94, Journ\'ees SMF, 24, Univ. Poitiers, Poitiers (1986). 



\bibitem{zedda} M. Zedda, {\em Canonical metrics on Cartan-Hartogs domains}, IJGMMP, 9 (1), (2012).

\end{thebibliography}
\end{document}